\documentclass{amsart}

\usepackage{amsmath,amsfonts,amssymb,amsthm,amscd,comment,euscript}
\usepackage{stmaryrd}
\usepackage[all]{xy}
\usepackage{graphicx}
\usepackage{enumerate} %allows to change enumeration items easily
\usepackage[colorlinks=true]{hyperref} %to generate pdf with links
\usepackage[usenames,dvipsnames]{xcolor}
\newcommand{\qq}{\mathbb{Q}}
\newcommand{\zz}{\mathbb{Z}}
\newcommand{\hh}{\mathtt{h}}
\newcommand{\DD}{\mathcal{D}(M)_F}
\newcommand{\RR}{R[[M]]_F}

\theoremstyle{plain}
\newtheorem{theo}{Theorem}

\theoremstyle{definition}
\newtheorem{defi}[theo]{Definition}
\newtheorem{rem}[theo]{Remark}
\newtheorem{example}[theo]{Example}

\begin{document}
%
%%%%%%%%%%%%%%%%%%%%%%%%%%%%%%%%%%%

\title{On formal Schubert polynomials}

\author{Kirill Zainoulline}


\begin{abstract}
Present notes can be viewed as an 
attempt to extend the notion of Schubert/Grothendieck polynomial of Lascoux-Sch\"utzenberger to the context of an arbitrary formal group law and of an arbitrary oriented cohomology theory.
\end{abstract}

\maketitle

Let $F\in R[[x,y]]$ be a commutative one-dimensional formal group law over a
commutative unital ring $R$ and let $\hh$ be an algebraic oriented
cohomology theory with the coefficient ring $\hh(pt)=R$. According to Levine-Morel \cite{levmor-book}
there is a 1-1 correspondence between such $F$'s and universal $\hh$'s. 
Indeed, given an oriented theory $\hh$ the respective
formal group law is determined by the Quillen formula for the first
charactersitic classes in the theory $\hh$ of the tensor product of line bundles
\[
c_1^\hh(L_1\otimes L_2)=F(c_1^\hh(L_1),c_1^\hh(L_2))
\] 
and given a formal
group law $F$ over $R$ one obtains the respective universal theory $\hh$ by tensoring with the algebraic cobordism $\Omega$,
i.e. \[
\hh(-):=\Omega(-)\otimes_{\Omega(pt)} R,\] 
where $\Omega(pt)\to R$ is obtained by specializing coefficients of
the universal formal group law.

For example, additive
formal group law $F_a(x,y)=x+y$ corresponds to the Chow theory $\hh=CH$,
multiplicative $F_m(x,y)=x+y-xy$ to the Grothendieck $\hh=K_0$ and the
universal one $F_u$ to the algebraic cobordism $\hh=\Omega$. Observe that in the
first two cases $\hh(pt)=\zz$ and in the last case the coefficient
ring is the Lazard ring which is infinitely generated over $\zz$.

\medskip

Let $G$ be a split semisimple linear group over a field $k$ containing a split
maximal torus $T$.
Following \cite[\S2]{CPZ} consider the formal group
algebra 
\[
\RR:=R[[x_\omega]]_{\omega\in
  M}/(x_0,x_{\omega+\omega'}-F(x_{\omega},x_{\omega'})),
\] 
where $M$ is the weight
lattice, together with the augmentation map $\epsilon\colon \RR \to
R$, $x_\omega \mapsto 0$. From the geometric point of view
$\RR$ models the completion of the
equivariant cohomology
$\hh_T(pt)$ and the map $\epsilon$ is the forgetful
map. Algebraically, $\RR$ is non-canonically isomorphic to the ring of
formal power series in $\mathop{rk}(M)$ variables.

Consider the algebra of formal
divided difference operators $\DD$ on $\RR$ and let
$\epsilon \DD^*:=Hom_R(\epsilon\circ \DD,R)$ denote the dual of the 
algebra of augmented operators. 
The main result of \cite[\S13]{CPZ} says that
if $\hh$ is a (weakly birationally invariant) oriented cohomology theory corresponding
to $F$ (e.g. $\hh=CH$, $K_0$ or $\Omega$), then there is an
$R$-algebra isomorhism 
\begin{equation}
\epsilon \DD^*\simeq \hh(X),
\end{equation}
where $X$ is the variety of Borel subgroups of $G$.
Moreover, it was shown that
the $R$-basis of $\hh(X)$ consisting of classes of Bott-Samelson resolutions of
Schubert varieties corresponds to the basis of $\epsilon \DD^*$ constructed as
follows:

First, for each element of the Weyl group $w\in W$ one chooses a reduced decomposition
$w=s_{i_1}s_{i_2}\ldots s_{i_r}$ into a product of simple reflections and denotes by $I_w=(i_1,\ldots,i_r)$
the respective reduced word.
Then one shows that the $R$-linear operators $\epsilon C_{I_w}^F$ defined by
composing $\epsilon$ with the
composite of the formal
divided difference operators $C_{I_w}^F=C_{i_1}^F\circ\ldots
\circ C_{i_r}^F$ 
form a basis of $\epsilon \DD$ \cite[Prop.~5.4]{CPZ}. 
Finally, the elements $A_{I_w}(z_0)$ give the desired basis, 
where $z_0$ is the element of $\epsilon \DD^*$ dual to $\epsilon
C_{I_{w_0}}^F(u_0)$ for some specially chosen $u_0\in (\ker\epsilon)^{\dim X}$, 
$w_0$ is the element of maximal length and $A_i$ is the operator on $\epsilon
\DD^*$ dual to the
operator on $\epsilon \DD$ given by composition on the right with
$C_i^F$ \cite[Thm.~13.13]{CPZ}.

\medskip

One of the major difficulties in extending the Schubert calculus to
such generalized theories $\hh$ (and, hence, the Schubert polynomials)
is the fact that
all mentioned bases are non-canonical, i.e. depend on
choices of reduced decompositions. Moreover, according to \cite{BE90}
they are canonical if and only if $F$ has the form
$F(u,v)=u+v-\beta uv$ for some $\beta\in R$. In other words, the
Bott-Samelson resolutions of Schubert varieties provide a canonical basis of $\hh(X)$ only for Chow groups ($\beta=0$),
Grothendieck $K_0$ ($\beta$ is invertible) and connective $K$-theory ($\beta\neq 0$ is non-invertible).

\medskip

In the present notes we try to overcome this difficulty and, hence,
provide a {\em canonical basis} of $\hh(X)$ by either
\begin{itemize}
\item[(1)] averaging
over all reduced decompositions, i.e. over all classes of
Bott-Samelson resolutions; or 
\item[(2)] exploiting the Kazhdan-Lusztig theory in the case of a
  special elliptic formal group law.
\end{itemize}

Observe that approach (1) works only after
inverting the Hurwitz numbers, e.g. over $\qq$, but over $\qq$ all formal group
laws become isomorphic. Therefore, one may suspect that we simply reduce to the known cases
of additive or multiplicative formal group laws. But this
isomorphism does not preserve the formal difference operators as well
as many other structures, so this is not the case.

Approach (2) seems to be even more interesting as it gives a canonical
basis integrally. However, we don't know how to extend it to other examples
of formal group laws.

\section{Averaging over reduced expressions}

Consider the evaluation map $\mathfrak{c}^F\colon \RR \to \epsilon
\DD^*$ of \cite[\S6]{CPZ}. 
Observe that on the level of cohomology (after identifying with
$\hh(X)$) 
it coincides with the characteristic map
induced by $\omega \mapsto c_1^\hh(L(\omega))$. In the case of the
additive or multiplicative formal group law it gives the characteristic
map described by Demazure in \cite{De74}, \cite{De73}. Moreover, according
to \cite[Thm.~6.9]{CPZ} if the Grothendieck torsion index $\tau$ of $G$ is
invertible (this always holds for Dynkin types $\mathrm{A}$ and
$\mathrm{C}$), 
then the kernel of $\mathfrak{c}^F$ is the ideal $\mathcal{I}_F^W$
generated by augmented $W$-invariant elements, and we obtain an $R$-algebra
isomorphism
\begin{equation}
\RR/\mathcal{I}_F^W \simeq \epsilon \DD^*,
\end{equation}
which in view of the results of \cite{HMSZ} and \cite{CZZ} relate the
invariant theory of $W$ with an $F$-version of the Hecke
ring of Kostant-Kumar \cite{KK86}, \cite{KK90}, \cite{Ku02} and
Bressler-Evens \cite{BE87}.
In general, though the kernel of
$\mathfrak{c}^F$ always contains the invariant ideal $\mathcal{I}^W$, the
induced map $\mathfrak{c}^F\colon \RR/\mathcal{I}_F^W \to \epsilon \DD^*$ is neither injective nor surjective.

Observe also that if $\tau$ is invertible, then we can identify the
basis $A_{I_w}(z_0)$ of $\epsilon \DD^*$ with the basis $C_{I_w^{rev}}(u_0)$
of $\RR/\mathcal{I}_F^W$, where $u_0=[pt]$
corresponds to the
class of a point (see \cite[Thm.~6.7]{CPZ}). The latter suggest to
define for each $w\in W$ the following element in $R_\qq[[M]]_F/\mathcal{I}_F^W$:
\begin{equation}
P_w^F:=\tfrac{1}{|red(w)|} \sum_{I_w\in red(w)} C_{I_w^{rev}}([pt]),
\end{equation}
where the sum is taken over the set $red(w)$ of all reduced words of $w$.

The results of \cite{HMSZ} and \cite{CZZ} then imply that
$\{P_w^F\}_{w\in W}$ is the desired canonical basis over $\qq$:

\begin{theo}
The elements $\{P_w^F\}_{w\in W}$ form a $R_{\qq}=R\otimes_\zz
\qq$-basis of $ \RR/\mathcal{I}_F^W$ and, hence, of $\hh(X)\otimes_\zz\qq$.
\end{theo}

\begin{proof}
By \cite[Prop.~5.8]{HMSZ} and \cite[Lem.~7.1]{CZZ} the difference
$(C_{i}C_{j})^{\circ m_{ij}}-(C_{j}C_{i})^{\circ m_{ij}}$ is a linear combination of
terms of length strictly smaller than $2m_{ij}$ (here $m_{ij}$ is the exponent
in the respective Coxeter relation).
Therefore,
each $P_w^F$ can be written
as $P_w^F=(C_{I_w^{rev}}+ (\text{products of smaller length}))([pt])$. So the matrix
expressing $P_w^F$ in terms of the usual basis
$\{C_{I_w^{rev}}\}_{w\in W}$ corresponding to a fixed
choice of reduced decompositions $\{I_w\}_{w\in W}$ is
upper-triangular with 1's on the main diagonal.
\end{proof}

Consider a root system of Dynkin type $\mathrm{A}_n$. Let
$\{e_1,\ldots,e_{n+1}\}$ be the standard basis with
$\alpha_i=e_i-e_{i+1}$ the set of simple roots. 
Consider a ring homomorphism
\[
R[t_1,t_2,\ldots,t_{n+1}] \to R[[\Lambda]]_F,\text{ given by
}t_i\mapsto x_{-e_i}.
\]
It is $S_{n+1}$-equivariant, therefore, it
induces a  map on quotients
\begin{equation}\label{isosch}
R[t_1,t_2,\ldots,t_{n+1}]/I \to R[[\Lambda]]_F/\mathcal{I}_F^W,
\end{equation}
where $I$ is the ideal generated by symmetric functions.
By Hornbostel-Kiritchenko \cite[Thm.~2.6]{HK} this is an $R$-algebra
isomorphism. 

\begin{defi}
We define
an $F$-Schubert polynomial $\pi^F_w$ to be the image of $P^F_w$ in
the quotient $R[t_1,\ldots, t_{n+1}]/I$ via the isomorphism \eqref{isosch}. 
\end{defi}

If $F$ has the form $F(u,v)=u+v-\beta uv$ for some $\beta\in
R$, i.e. exactly the formal group law for which the respective composites
$C_{I_w}$ do not depend on choices of reduced words of $w$,
then for $\beta=0$
(resp. for $\beta=1$) $\pi^F_w$ coincide with the respective
Schubert (resp. Grothendieck) polynomials of
Lascoux-Sch\"utzenberger (e.g. see \cite{Fo94}, \cite{FK}) and for arbitrary $\beta$ we obtain polynomials
studied in \cite{FK} and \cite{Hu12}.

\medskip

Observe that
under this isomorphism the class of the point $[pt]$ corresponds to
the class of the polynomial $t_1^nt_2^{n-1}\ldots
t_n$
and the formal divided difference operator $C_i(u)=\tfrac{u}{x_{-i}}+\tfrac{s_i(u)}{x_i}=(1+s_i)(\tfrac{u}{x_{-i}})$
(here $x_{-i}=x_{-\alpha_i}=x_{e_{i+1}-e_{i}}$)
corresponds to the operator 
\begin{equation}
C_i(f)=(1+\sigma_i)(\tfrac{f}{\rho_i}),
\end{equation}
where $\sigma_i$ swaps $t_i$ and $t_{i+1}$ and
$\rho_i$ is the formal power series given by $t_i-_F t_{i+1}=F(t_i,\imath(t_{i+1}))$, where $\imath(x)$ is the formal
inverse series of $x$.

By definition, operators $C_i$ are $\RR^{W_i}$-linear (here
$W_i=\langle s_i\rangle$) and satisfy \cite[Prop.~3.13]{CPZ}:
\[
C_i(uv)=C_i(u)v+s_i(u)C_i(v)-\kappa_is_i(u)v\text{
and }C_i(1)=\kappa_i,
\]
where $\kappa_i=\tfrac{1}{\rho_i}+\tfrac{1}{\imath(\rho_i)}\in
\RR$.
We also have
$C_i(t_jf)=t_jC_i(f)$ for $j\neq i,i+1$ and
\[
C_i(t_i)=\tfrac{t_i}{t_i-_F t_{i+1}}+\tfrac{t_{i+1}}{t_{i+1}-_F
  t_i}=\tfrac{t_i-_F t_{i+1}+_F t_{i+1}}{t_i-_F
  t_{i+1}}+\tfrac{t_{i+1}}{t_{i+1}-_F
  t_i}
=\tfrac{F(\rho_i,t_{i+1})-t_{i+1}}{\rho_i}+t_{i+1}\kappa_i.
\]
Using these formulas one can compute the
polynomials $\pi_w^F$.

Our goal now is (using these formulas) to express each polynomial
$\pi_w^F$ 
as a linear combination of sub-monomials of $t_1^nt_2^{n-1}\ldots t_n$
with coefficients from $R$.

\begin{example}
We can write an arbitrary formal group law $F$ as 
\[
F(x,y)=x+y-xyg(x,y).
\]
This implies that 
\[
\tfrac{1}{z}+\tfrac{1}{\imath(z)}=g(z,\imath(z)).
\]
If $F(x,y)=x+y+a_{11}xy+a_{12}xy(x+y)+O(4)$, then
\[
g(x,y)=-a_{11}-a_{12}(x+y)-a_{13}(x^2+y^2)-a_{22}xy+O(3)
\]
and
\[
\imath(x)=-x+a_{11} x^2 - a_{11}^2 x^3 +O(4).
\]
So, after substituting, we obtain
\[
x-_F y=x + (-y+a_{11} y^2 - a_{11}^2 y^3) + a_{11} x (-y+a_{11} y^2) + a_{12}x(-y)(x-y)+
O(4)=
\]
\[
=(x-y) - a_{11}y(x-y)+a_{11}^2y^2(x-y)-a_{12}xy(x-y)+ O(4),
\]
and, hence,
\[
(x-_F y)+(y-_F x)=a_{11}(x-y)^2-a_{11}^2(x+y)(x-y)^2+O(4),
\]
\[
(x-_F y)^2+(y-_F x)^2=2(x-y)^2-2a_{11}(x+y)(x-y)^2+O(4),
\]
\[
(x-_F y)(y-_F x)=-(x-y)^2+a_{11}(x+y)(x-y)^2+O(4).
\]
Combining these together we get
\[
g(x-_F y, y-_F x)=-a_{11}-a_{12}a_{11}(x-y)^2-2a_{13}(x-y)^2+a_{22}(x-y)^2+O(3)=
\]
\[
-a_{11}-(a_{11}a_{12}+2a_{13}-a_{22})(x-y)^2+O(3).
\]
We then obtain
\[
r(x,y)=\tfrac{F(x,y)-y}{x}=1+a_{11}y+a_{12}y(x+y)+a_{13}y(x^2+y^2)+a_{22}xy^2+O(4).
\]
Hence,
\[
r(x-_F y,y)=1+a_{11}y+a_{12}y(x-
a_{11}y(x-y))+a_{13}y((x-y)^2+y^2)+a_{22}(x-y)y^2+O(4)=
\]
\[
=1+a_{11}y+a_{12}xy+(a_{22}-a_{12}a_{11})(x-y)y^2+a_{13}y((x-y)^2+y^2)+O(4),
\]
and, therefore,
\[
r(x-_F y,y)+yg(x-_F y,y-_F x)=1+a_{12}xy+(a_{22}-a_{12}a_{11})xy(x-y)+a_{13}xy(2y-x)+O(4).
\]
Observe that there is a relation $2(a_{22}-a_{11}a_{12})=3a_{13}$ in
the Lazard ring (the only relation in degree 4)
which gives
\[
2[(a_{22}-a_{12}a_{11})xy(x-y)+a_{13}xy(2y-x)]=2a_{13}xy(x+y).
\]
In particular, if $2$ is invertible and $x=t_i$, $y=t_{i+1}$ for the type $\mathrm{A}_2$, then
$t_it_{i+1}(t_i+t_{i+1})$ is in the ideal $I$ generated by symmetric
functions in $t_1,t_2,t_3$, meaning that for $i=1,2$
\[
C_i(t_i)=\tfrac{t_i}{t_i-_F t_{i+1}}+\tfrac{t_{i+1}}{t_{i+1}-_F t_i}=r(t_i-_F t_{i+1},t_{i+1})+t_{i+1}\kappa_i=1+a_{12}t_it_{i+1}.
\]
\end{example}

\begin{example} Using the formulas above
we
obtain the following expressions for $\pi^F_w$ in the
$\mathrm{A}_2$-case for an arbitrary $F$
(this agrees with computations at the end of \cite{HK} and \cite{CPZ})
$$
\pi_1^F=C_1([pt])=t_1t_2,\qquad
\pi_2^F=C_2([pt])=t_1^2,
$$

Indeed, 
\[
C_1(t_1^2t_2)=\tfrac{t_1^2t_2}{t_1-_F t_2}+\tfrac{t_1t_2^2}{t_2-_F
  t_1}=t_1t_2C_1(t_1)=t_1t_2
\]
and
\[
C_2(t_1^2t_2)=\tfrac{t_1^2t_2}{t_2-_F t_3} +\tfrac{t_1^2t_3}{t_3-_F
  t_2}=t_1^2(\tfrac{t_2}{t_2-_F t_3}+\tfrac{t_3}{t_3-_F
  t_2})=t_1^2(1+a_{12}t_2t_3)=t_1^2.
\]
$$
\pi_{21}^F=C_{12}([pt])=t_1+t_2+a_{11}\pi_1^F,\qquad
\pi^F_{12}=C_{21}([pt])=C_2(t_1t_2)=t_1,
$$
Indeed,
\[
C_1(t_1^2)=C_1(t_1)t_1+t_2C_1(t_1)-\kappa_1 t_1t_2=(1+a_{12}t_1t_2)(t_1+t_2)+a_{11}t_1t_2=t_1+t_2+a_{11}t_1t_2
\]
and
\[
C_2(t_1t_2)=t_1C_2(t_2)=t_1(1+a_{12}t_2t_3)=t_1
\]

And for the element of maximal length $w_0=(121)=(212)$ we obtain
$$
C_{212}([pt])=1+a_{12} \pi_2^F,\qquad
C_{121}([pt])=C_1(t_1)=1+a_{12} t_1 t_2=1+a_{12}\pi_1^F.
$$

Indeed,
\[
C_2(t_1+t_2+a_{11}t_1t_2)=t_1C_2(1)+C_2(t_2)+a_{11}t_1C_2(t_2)=
\]
\[
=t_1(-a_{11}-(a_{11}a_{12}+2a_{13}-a_{22})(t_2-t_3)^2)+1+a_{12}t_2t_3+a_{11}t_1(1+a_{12}t_2t_3)=
\]
\[
1+a_{12}t_2t_3-\tfrac{1}{2}a_{13}t_1(t_2-t_3)^2=1+a_{12}t_2t_3=1+a_{12}t_1^2
\]
as $t_1^2\equiv t_2t_3$ and $t_1(t_2-t_3)^2$ is in the ideal ($t_1(t_2-t_3)^2\equiv
t_1(t_2^2+t_3^2)\equiv t_1^3 \equiv 0$).

Therefore,
\[
\pi_{w_0}^F=1+\tfrac{1}{2}a_{12}(t_1^2+t_1t_2).
\]

Observe that the twisted braid relation (which leads to the
dependence on choices) of \cite[Prop. 5.8]{HMSZ} then coincides with
\[
C_{121}-a_{12} C_1=C_{212}-a_{12} C_2.
\]
\end{example}

\section{A special elliptic formal group law and the Kazhdan-Lustig basis}

Consider an elliptic curve given in Tate coordinates by 
\[(1-\mu_1 x-\mu_2 x^2)y=x^3.
\] 
The corresponding formal group law over the coefficient ring
$R=\mathbb{Z}[\mu_1,\mu_2]$ is given by (e.g. \cite[Example~63]{BB10}), 
\[
F(x,y):=\tfrac{x+y-\mu_1xy}{1+\mu_2xy}\]
and will be called a {\em special elliptic} formal group law.
Observe that by definition, we have
\[
F(x,y)=x+y-xy(\mu_1+\mu_2F(x,y)),\text{ so } a_{11}=-\mu_1\text{ and }a_{12}=-\mu_2.
\]

\medskip

By \cite[Theorem~5.14]{HMSZ}
for the type $A_n$ the
algebra $\DD$ is generated by operators  $C_i$, $i\in 1..n$,
and multiplications by elements $u\in \RR$ subject to the folowing relations:
\begin{itemize}
\item[(a)] $C_i^2=\mu_1 C_i$ 
\item[(b)] $C_{ij}=C_{ji}$ for $|i-j|>1$,
\item[(c)] $C_{iji} - C_{jij}=\mu_2(C_j-C_i)$ for $|i-j|=1$ and
\item[(d)] $C_iu=s_i(u)C_i+\mu_1u-C_i(u)$,
\end{itemize}

\medskip

Recall that the Iwahori-Hecke algebra $\mathcal{H}$ of the
symmetric group $S_{n+1}$ is
(after the respective normalization)
an $\mathbb{Z}[t,t^{-1}]$-algebra with generators $T_i$, $i=1..n$, subject to
the following relations:
\begin{itemize}
\item[(A)] $(T_i-t^{-1})(T_i+t)=0$ or, equivalently, $T_i^2=(t^{-1}-t)T_i+1$,
\item[(B)] $T_{ij}=T_{ji}$ for $|i-j|>1$ and 
\item[(C)] $T_{iji} =T_{jij}$ for $|i-j|=1$.
\end{itemize}

Observe that $T_i$'s appearing in the classical definition of the Iwahori-Hecke
algebra in \cite[Def.~7.1.1]{CG10} correspond to $tT_i$ in our notation, where $t=q^{-1/2}$.

\medskip

Following \cite[Def.~5.3]{HMSZ} let $\mathrm{D}_F$ denote the $R$-subalgebra of $\DD$ generated by
the elements $C_i$, $i=1..n$, only. In
\cite[Prop.~6.1]{HMSZ} it was shown that for $F=F_a$
(resp. $F=F_m$) $\mathrm{D}_F$ is isomorphic to the
nil-Hecke algebra (resp. the 0-Hecke algebra) of Kostant-Kumar. 

\medskip

Comparing the relations for $\mathrm{D}_F$ and $\mathcal{H}$ we see that
for $R=\zz[t,t^{-1}][\tfrac{1}{t+t^{-1}}]$, $\mu_1=1$ and $\mu_2=-\tfrac{1}{(t+t^{-1})^2}$
there is an isomorphism of $R$-algebras (see \cite{CZZ1} and
\cite{LNZ} for the case of an arbitrary root system)
\begin{equation}\label{Heckiso}
\mathcal{H}[\tfrac{1}{t+t^{-1}}]\simeq\mathrm{D}_F\quad\text{ given on generators by }\;
T_i\mapsto (t+t^{-1})C_i-t,\; i=1..n.
\end{equation}

By definition of \eqref{Heckiso} the involution on $\mathcal{H}$ (sending $t\mapsto t^{-1}$
and $T_i\mapsto T_i^{-1}$) corresponds to the
involution on $\mathrm{D}_F$ obtained by extending the involution
$t\mapsto t^{-1}$ on the coefficient ring. Observe that each
push-pull element $C_i=\tfrac{1}{t+t^{-1}}(T_i+t)$ is invariant under this involution.

\medskip

Consider the Kazhdan-Lusztig basis $\{C_w'\}_{w\in W}$ on $\mathcal{H}$
(e.g. see \cite{CG10}). Recall that it is unique and {\em does not depend
on choices} of reduced decompositions.
After the respective normalization we have
\[
C_w'=T_w+\sum_{v<w}t\pi_{v,w}(t) T_v,
\]
where $\deg \pi_{v,w}\le l(w)-l(v)-1$ and $\pi_{v,w}$ are the
Kazhdan-Lusztig polynomials.
For instance, $C_i'=T_i+t$, $C_{ij}'=T_{ij}+t(T_i+T_j)+t^2$ and $C_{iji}'=T_{iji}+t(T_{ij}+T_{ji})+t^2(T_i+T_j)+t^3$.

Let $C_w$ denote the element in $\mathrm{D}_F$ that corresponds to
$C_w'$ via \eqref{Heckiso}. Choose a reduced word $I_w$ for each
$w\in W$. Then 
\[
C_w=(t+t^{-1})^{l(w)}(C_{I_w}+lower\, degree\, terms)+
\]
\[+\sum_{v<w} t
\pi_{v,w}(t)(t+t^{-1})^{l(v)}(C_{I_v}+{lower\, degree\, terms}) 
\]
where the right hand side does not depend on choices of  reduced
decompositions. This suggests the following

\begin{defi}
We define the special elliptic polynomial $\pi_w^{se}$ to be
the image  in $\zz[t,t^{-1}][t_1,\ldots,t_{n+1}]/I$ of the element 
$\tfrac{1}{(t+t^{-1})^{l(w)}}C_{w^{-1}}([pt])$ via
\eqref{isosch}. 
\end{defi}

We expect polynomials $\pi_w^{se}$ to play the same role (in the
special elliptic case) as the
Schubert (resp. Grothendieck) polynomials for Chow groups
(resp. $K_0$).

\begin{example} For the type $A_2$ we obtain
\[
\pi_i^{se}=C_i([pt]),\quad \pi^{se}_{ij}=C_jC_i([pt])
\]
and for the element of maximal length we obtain exactly the twisted
braid relation
\[
\pi_{w_0}^{se}=(C_{121}+\mu_2 C_1)([pt])=(C_{212}+\mu_2 C_2)([pt])=1.
\]
\end{example}

\begin{rem}
It would be interesting to see
\begin{enumerate}
\item that $\pi_{w_0}^{se}=1$, for the element of maximal length $w_0$.
\item
whether $\pi_w^{se}$ corresponds to the class of an
{\em actual} resolution of the respective Schubert
variety $X_w$.
\end{enumerate}
\end{rem}

\end{document}